\documentclass[11pt]{article}
\usepackage[T1,T2A]{fontenc}
\usepackage[english,russian]{babel}
\usepackage{amsmath, amssymb, amsthm, amscd}
\usepackage{amsfonts}
\usepackage{float}
\usepackage{array}

\theoremstyle{plain}
\newtheorem{lemma}{Lemma}
\newtheorem{theorem}{Theorem}
\theoremstyle{remark}
\newtheorem{remark}{Remark}

\usepackage[a4paper, left=1.5cm, right=1.5cm, top=2cm, bottom=2cm]{geometry}

\title{Asymptotic probability of a fixed edge being on the boundary of the convex hull of a random walk in $\mathbb{Z}^2$}
\author{A.\,O.~Mysliuk\thanks{Mysliuk Aleksandr --- Postgraduate, Lomonosov Moscow State University, Faculty of Mechanics and Mathematics, Chair of Probability Theory, e-mail: aleksander.mysluk@gmail.com}}
\date{}

\begin{document}
\selectlanguage{english}
\maketitle


\vspace{0.5cm}
\textbf{Abstract.} A simple symmetric random walk in the space $\mathbb{Z}^2$ is considered. The asymptotic behavior as the number of jumps tends to infinity of the probability that a fixed edge of the random walk lies in the polygon that forms the boundary of the convex hull is investigated.


\section{Introduction and statement of the problem}

Asymptotic properties of convex hulls of random sets of points have been actively studied since the 1960s. Main results can be found in the works of G.~Baxter [1,2], D.~Zaporozhets and V.~Vysotsky [3], A.~Rényi and R.~Sulanke [4,5], F.~Spitzer and H.~Widom [6,7], B.~Efron [8], Sparre Andersen [11,12] and many others. In the cited works, random walks in $\mathbb{R}^d$ are considered. We have not found any works devoted to the study of convex hulls of random walks on multidimensional lattices. Despite the fact that a walk in $\mathbb{Z}^d$ is a particular case of a walk in $\mathbb{R}^d$, a direct transfer of results from $\mathbb{R}^d$ to $\mathbb{Z}^d$ is not always possible, since the main necessary requirements imposed on the jumps of the random walk may be violated. One of the main constraints used for jumps is the continuity of the distribution, which is immediately violated on lattices. Some results can be transferred to $\mathbb{Z}^d$. For example, in [10] limit theorems for convex hulls are formulated under assumptions that allow considering the lattice case. However, many combinatorial-geometric characteristics of the convex hull cannot be directly transferred. Such characteristics include, for instance, the number of edges on the boundary of the convex hull or the asymptotic probability that an edge belongs to the boundary, which is the subject of the present work.

\section{Main results}

We consider a simple symmetric random walk $S_n = \sum_{k=1}^n X_k$, $S_0:=\boldsymbol{0}$ on the two-dimensional lattice $\mathbb{Z}^2$. The jumps $X_i$ are independent and identically distributed according to the law
\[
\mathbb{P}(X_i = (1,0)) = \mathbb{P}(X_i = (-1,0)) = \mathbb{P}(X_i = (0,1)) = \mathbb{P}(X_i = (0,-1)) = \frac14.
\]

Let $C_n$ denote the convex hull of the points $S_0,\dots,S_n$. Let $\partial C_n$ denote the boundary of the convex hull $C_n$, which is a convex polygon. Any two adjacent vertices of $\partial C_n$ correspond to two partial sums of the random walk $S_{n_1}$ and $S_{n_2}$, $0\leq n_1 < n_2 \leq n$.
In this work we investigate the asymptotic behavior of the probability $\mathbb{P}(S_{n_1}S_{n_2}\in\partial C_n)$ for fixed $n_1,n_2$ and $n\to\infty$.

To prove the main result, we need the following lemma for a random walk on $\mathbb{Z}$.

\begin{lemma}
Let $k,m$ be two positive integers, at least one of them nonzero. $S_n = \sum_{i=1}^n X_i$, $S_0:=\boldsymbol{0}$, where the one-dimensional jumps $X_i$ are independent, identically distributed according to the law:
\[
\mathbb{P}(X_i = k) = \mathbb{P}(X_i = -k) = \mathbb{P}(X_i = m) = \mathbb{P}(X_i = -m) = \frac14.
\]
Let $N_n$ denote the number of positive partial sums among $S_1,\dots,S_n$, and set $p_n = \mathbb{P}(N_n = n)$, $p_0 = 1$. Then the asymptotic equality holds
\[
p_n \sim \frac{C_{k,m}}{\sqrt{\pi n}},\quad n\to\infty,
\]
where
\[
C_{k,m} = \exp\left(\frac1{4\pi}\int_{-\pi}^{\pi}\ln\bigl(1 - \tfrac12(\cos k\theta + \cos m\theta)\bigr)\,d\theta\right).
\]
\end{lemma}

\begin{proof}
Let $P(s) = \sum_{n=0}^\infty p_n s^n$ denote the generating function of the sequence $p_n$. For convenience set $a_n = \mathbb{P}(S_n >0)$, $h_n = \mathbb{P}(S_n = 0)$. Let $A(s) = \sum_{n=1}^\infty a_n s^n$, $H(s) = \sum_{n=1}^\infty h_n s^n$ denote the corresponding generating functions of the sequences $a_n$ and $h_n$. By Theorem 1 of [12] for $|s|<1$ we have
\begin{equation}
P(s) =\exp\left({\int_0^s \frac{A(\sigma)}{\sigma}\,d\sigma}\right).
\end{equation}
Let's find $A(s)$. By the symmetry of the random walk we have
\[
a_n = \mathbb{P}(S_n >0) = \frac12(1 - \mathbb{P}(S_n=0)) = \frac12(1-h_n),
\]
which yields a relation for the generating functions. For $|s|<1$ we have
\begin{equation}
A(s) = \frac12\left(\frac{1}{1-s} - H(s)\right).
\end{equation}
To compute $h_n$, one must count the number of paths ending at $0$ and divide by the total number of paths $4^n$. This number of paths equals the number of ways to represent $0$ as a sum of $n$ summands $\pm k$, $\pm m$, which equals the coefficient of $t^0$ after expanding $(t^{-k} + t^{-m} + t^{m} + t^{k})^n$. Denote $f(t):=(t^{-k} + t^{-m} + t^{m} + t^{k})$. Thus $H(s)$ takes the form
\[
H(s) = \sum_{n=1}^\infty ([t^0]f(t)^n)\left(\frac{s}4\right)^n,
\]
where $[t^0]f(t)^n$ denotes the coefficient of $t^0$ in $f(t)^n$. This coefficient can be computed using Cauchy's integral formula
\[
[t^0]f(t)^n = \frac1{2\pi i}\oint_{|t| = 1} \frac{f(t)^n}{t}\,dt,
\]
hence we obtain
\[
H(s) = \sum_{n=1}^\infty \left(\frac{s}4\right)^n\frac{1}{2\pi i}\oint_{|t| = 1} \frac{f(t)^n}{t}\,dt.
\]
For small values of $s$ ($|s|<1$) we can interchange integration and summation in the standard way and apply the formula for the sum of a geometric series
\[
H(s) = \frac1{2\pi i}\oint_{|t| = 1}\frac1t\left(\sum_{n=0}^\infty \left(\frac{sf(t)}4\right)^n\right)dt = \frac1{2\pi i} \oint_{|t| = 1} \frac1t\frac{1}{\left(1-\frac{sf(t)}4\right)}\, dt.
\]
After the substitution $t=e^{i\theta}$ we get
\[
H(s)=\frac1{2\pi}\int_{-\pi}^{\pi}\frac{d\theta}{1 - \frac{s}2(\cos k\theta + \cos m\theta)}.
\]
Substituting $H(s)$ into (2) yields the expression for $A(s)$
\[
A(s) = \frac12\left(\frac{1}{1-s} - \frac1{2\pi}\int_{-\pi}^{\pi}\frac{d\theta}{1 - \frac{s}2(\cos k\theta + \cos m\theta)}\right).
\]
Now compute
\begin{align*}
\int_{0}^s \frac{A(\sigma)}{\sigma}\,d\sigma &= \frac12\int_0^s\left(\frac1{\sigma(1-\sigma)} - \frac{H(\sigma)}{\sigma}\right)d\sigma \\
&= \frac12\int_0^s\left(\frac1\sigma - \frac1\sigma + \frac1{1-\sigma} - \frac{H(\sigma) - 1}{\sigma}\right)d\sigma \\
&= \frac12\int_0^s\frac{d\sigma}{1-\sigma} - \frac12\int_0^s\frac{H(\sigma)-1}{\sigma}\,d\sigma.
\end{align*}
For $0<s<1$ from (1) we obtain the final result
\begin{align*}
P(s) &= \exp\left(\int_0^s \frac{A(\sigma)}{\sigma}d\sigma\right) \\
&= \exp\left(-\frac12\ln(1-s) - \frac{1}{4\pi}\int_0^s\int_{-\pi}^{\pi}\frac{1}{\sigma}\left(\frac{1}{1 - \frac{\sigma}2(\cos k\theta + \cos m\theta)} - 1\right)d\theta\,d\sigma\right) \\
&= \exp\left(-\frac12\ln(1-s) - \frac{1}{4\pi}\int_{-\pi}^{\pi}\int_0^s\frac{\cos k\theta + \cos m\theta}{2-\sigma(\cos k\theta + \cos m\theta)}\,d\sigma\,d\theta\right) \\
&= \exp\left(-\frac12\ln(1-s) + \frac{1}{4\pi}\int_{-\pi}^{\pi}\ln\!\Bigl(1-\frac{s}2(\cos k\theta + \cos m\theta)\Bigr)d\theta\right).
\end{align*}
Thus,
\[
P(s) = \frac{1}{\sqrt{1-s}}\exp\left(\frac{1}{4\pi}\int_{-\pi}^{\pi}\ln\!\Bigl(1-\frac{s}2(\cos k\theta + \cos m\theta)\Bigr)d\theta\right),\qquad 0<s<1.
\]
Letting $s\to1-$ we obtain
\[
p_n \sim \frac{C_{k,m}}{\sqrt{\pi n}},\quad n\to\infty,
\]
where $C_{k,m} = \exp\left(\frac1{4\pi}\int_{-\pi}^{\pi}\ln\bigl(1 - \tfrac12(\cos k\theta + \cos m\theta)\bigr)d\theta\right)$. The lemma is proved.
\end{proof}

\begin{remark}
The formula for $C_{k,m}$ remains valid when $k=m$. In that case we obtain the well-known result for a walk with steps $\pm1$: $p_n\sim\frac{1}{\sqrt{2\pi n}},\ n\to\infty$.
\end{remark}

Now we prove the main result for the two-dimensional simple symmetric random walk described earlier.

\begin{theorem}
Let $n_1,n_2$ be fixed time instants $0\leq n_1<n_2 \leq n$. Then as $n\to\infty$ the asymptotic equality holds
\[
\mathbb{P}(S_{n_1}S_{n_2}\in\partial C_n) \sim \frac{C_{n_1,n_2}}{\sqrt{n}},\quad n\to\infty,
\]
where $C_{n_1,n_2}$ is some positive constant.
\end{theorem}

\begin{proof}
Let $X_1^{(k,m)},\dots,X_n^{(k,m)}$ denote the projections of the jumps $X_1,X_2,\dots,X_n$ onto the line $y=\frac{k}{m}x$ orthogonal to the segment $S_{n_1}S_{n_2}$. Let $S_0^{(k,m)},\dots,S_n^{(k,m)}$ denote the projections of the partial sums $S_0,\dots,S_n$. Up to a constant factor, $X_i^{(k,m)}$ has the distribution
\[
\mathbb{P}(X_i^{(k,m)} = k) = \mathbb{P}(X_i^{(k,m)} = -k) = \mathbb{P}(X_i^{(k,m)} = m) = \mathbb{P}(X_i^{(k,m)} = -m) = \frac14.
\]

Introduce a coordinate axis on the line $y=\frac{k}{m}x$ such that the projections $S_{n_1}^{(k,m)}$ and $S_{n_2}^{(k,m)}$ are equal to $0$. If the edge $S_{n_1}S_{n_2}$ is an edge of $\partial C_n$, then the projections $S_0^{(k,m)},\dots,S_{n_1-1}^{(k,m)},S_{n_1+1}^{(k,m)},\dots,S_{n_2-1}^{(k,m)},S_{n_2+1}^{(k,m)},\dots,S_n^{(k,m)}$ have the same sign. By symmetry, without loss of generality, we consider positive sums. Thus,
\[
\mathbb{P}(S_{n_1}S_{n_2}\in\partial C_n \mid S_{n_1}S_{n_2}\parallel(k,m)) = 2\cdot\mathbb{P}_{0,n_1}^{(k,m)}\cdot\mathbb{P}_{n_1,n_2}^{(k,m)}\cdot\mathbb{P}_{n_2,n}^{(k,m)},
\]
where
\begin{align*}
\mathbb{P}_{0,n_1}^{(k,m)} &= \mathbb{P}(S_0^{(k,m)}>0,\dots, S_{n_1-1}^{(k,m)} >0 \mid S_{n_1}^{(k,m)} = 0),\\
\mathbb{P}_{n_1,n_2}^{(k,m)} &= \mathbb{P}(S_{n_1+1}^{(k,m)}>0,\dots, S_{n_2-1}^{(k,m)} >0 \mid S_{n_1}^{(k,m)} = S_{n_2}^{(k,m)} = 0),\\
\mathbb{P}_{n_2,n}^{(k,m)} &= \mathbb{P}(S_{n_2+1}^{(k,m)}>0,\dots, S_{n}^{(k,m)} >0 \mid S_{n_2}^{(k,m)} = 0).
\end{align*}
By Lemma 1 we obtain
\begin{equation}
\mathbb{P}_{n_2,n}^{(k,m)}\sim \frac{C_{k,m}}{\sqrt{\pi (n-n_2)}},\quad n\to\infty.
\end{equation}

Let $\Omega_{n_1,n_2}$ denote the set of possible slope coefficients $(k,m)$ of the vector $S_{n_1}S_{n_2}$. Denote by $F(x,y)$ the distribution function of the possible pairs $(k,m)$. Integrating the conditional probability over the set of possible pairs and applying (3), we have
\begin{align*}
\mathbb{P}(S_{n_1}S_{n_2}\in\partial C_n) &= \int_{\Omega_{n_1,n_2}} \mathbb{P}(S_{n_1}S_{n_2}\in\partial C_n \mid S_{n_1}S_{n_2}\parallel(k,m))\, dF(k,m) \\
&= \int_{\Omega_{n_1,n_2}} \frac{2\cdot\mathbb{P}_{0,n_1}^{(k,m)} \cdot \mathbb{P}_{n_1,n_2}^{(k,m)}\cdot C_{k,m}}{\sqrt{\pi(n-n_2)}}\, dF(k,m) \\
&\sim \int_{\Omega_{n_1,n_2}} \frac{2\cdot\mathbb{P}_{0,n_1}^{(k,m)} \cdot \mathbb{P}_{n_1,n_2}^{(k,m)}\cdot C_{k,m}}{\sqrt{n\pi}}\, dF(k,m),\quad n\to\infty.
\end{align*}
Hence it follows that
\[
\mathbb{P}(S_{n_1}S_{n_2}\in\partial C_n) \sim \frac{C_{n_1,n_2}}{\sqrt{n}},\quad n\to\infty,
\]
where $C_{n_1,n_2} = \int_{\Omega_{n_1,n_2}} \frac{2}{\sqrt{\pi}}\,\mathbb{P}_{0,n_1}^{(k,m)}\mathbb{P}_{n_1,n_2}^{(k,m)} C_{k,m}\, dF(k,m)$. Theorem 1 is proved.
\end{proof}

\section{Numerical simulation}

To confirm the analytical results, simulations were performed. For each number of steps $n$ the Monte Carlo method was used with $1\,000\,000$ runs. The table shows the ratio of the probability $\hat{\mathbb{P}}(N_n = n)$ obtained in the simulation to the analytical asymptotics. The theoretical and experimental values of the constant $C_{k,m}$ are also given.

\begin{table}[h]
\centering
\caption{$k = 1$, $m = 1$}
\begin{tabular}{|c|c|c|c|}
\hline
$n$ & $\hat{\mathbb{P}}(N_n = n) / \frac{C_{k,m}}{\sqrt{\pi n}}$ & $C_{k,m}$ theor. & $C_{k,m}$ sim.\\
\hline
100 & 0.9982 & 0.7072 & 0.7059\\
\hline
1000 & 0.9975 & 0.7072 & 0.7054\\
\hline
10000 & 1.0178 & 0.7072 & 0.7198\\
\hline
\end{tabular}
\end{table}

\begin{table}[h]
\centering
\caption{$k = 2$, $m = 1$}
\begin{tabular}{|c|c|c|c|}
\hline
$n$ & $\hat{\mathbb{P}}(N_n = n) / \frac{C_{k,m}}{\sqrt{\pi n}}$ & $C_{k,m}$ theor. & $C_{k,m}$ sim.\\
\hline
100 & 0.9960 & 0.8091 & 0.8058\\
\hline
1000 & 1.0006 & 0.8091 & 0.8096\\
\hline
10000 & 0.9913 & 0.8091 & 0.8020\\
\hline
\end{tabular}
\end{table}

\begin{table}[H]
\centering
\caption{$k = 5$, $m = 7$}
\begin{tabular}{|c|c|c|c|}
\hline
$n$ & $\hat{\mathbb{P}}(N_n = n) / \frac{C_{k,m}}{\sqrt{\pi n}}$ & $C_{k,m}$ theor. & $C_{k,m}$ sim.\\
\hline
100 & 0.9959 & 0.8896 & 0.8859\\
\hline
1000 & 1.0063 & 0.8896 & 0.8952\\
\hline
10000 & 1.0072 & 0.8896 & 0.8960\\
\hline
\end{tabular}
\end{table}

\section{Conclusions}

The asymptotic probability of a fixed edge being on the boundary of the convex hull has been obtained. The probability estimate obtained in the simulations is very close to the analytical result, which confirms its validity.

The author expresses gratitude to Professor E.B. Yarovaya for discussing the problem.

\end{document}